\documentclass{amsart}

\usepackage{amsmath,amsfonts,amssymb}
\usepackage{graphicx,psfrag,subfigure}

\newtheorem{thm}{Theorem}
\newtheorem{rk}{Remark}
\newtheorem{prop}{Proposition}

\newtheorem{lemma}{Lemma}
\newtheorem{defi}{Definition}

\newcommand{\R}{{\mathbb{R}}}

\newcommand{\N}{{\mathbb{N}}}

\DeclareMathOperator{\fix}{Fix}

\begin{document}
\title{Sphere branched coverings and the growth rate inequality}
\author{J.Iglesias, A.Portela, A.Rovella and J.Xavier}

\begin{abstract}  We show that the growth inequality rate $$\limsup \frac{1}{n} \log (\# \fix (f^n))\geq \log d$$  holds for branched coverings of degree $d$ of the sphere $S^2$ having a 
completely invariant simply connected region $R$ with locally connected boundary, except in some degenerate cases with known couterexamples.

\end{abstract}

\maketitle

\section{Introduction}
This paper deals with the following open problem:  let $f: S^2 \to S^2$ be a continuous map of degree $d$, $|d|>1$, and let $N_nf$ denote the number of fixed points of $f^n$. When does 
the growth rate inequality $\limsup \frac{1}{n} \log N_nf\geq \log d$ hold for $f$? (This is Problem 3 posed in \cite{shub2}).

It is known that this inequality does not always hold. The simplest example is the map expressed in polar coordinates in $\R^2$ as 
$(r,\theta)\to (dr,d\theta)$ and extended to the sphere with $\infty\to \infty$. It has degree $d$ and just two periodic points. In \cite{iprx2} other examples are presented, where
the nonwandering set is not reduced to the set of periodic points.  On the other hand, the inequality is known to hold if $f$ is a rational map \cite{shub}, if $f$ is $C^1$ and preserves
the latitude foliation \cite{ps} and \cite{mis}, if the critical points form a two-periodic cycle \cite{iprx3}, and if all periodic orbits are isolated as invariant sets and $f$ has no sources of degree
$r$, $|r|\geq 1$ \cite{lhc}.  Whenever the growth rate inequality holds, we say that $f$ {\it has the rate}.

We will work with branched coverings and  make the following assumption: there exists $R\subset S^2$  connected, open, proper (i.e. a {\it region}), simply connected and completely invariant 
($f^ {-1} (R) = R$).  
This allows the localization of the set where the periodic points live:  the boundary of $R$. These two assumptions (branched covering + completely invariant simply connected region) are 
strong assumptions,
but there will be more, as all known examples of maps not having the rate satisfy these two assumptions. These examples also satisfy that there are exactly two fixed critical points of 
multiplicity $d-1$, one in $R$ and the other one in the boundary of $R$.   It follows that $f$ can be thouhgt as a covering map of the open annulus $\R ^2 \backslash \{0\}$.  Related work in the annulus was developed in the 
sequel \cite{iprx1}, \cite{iprx2}, \cite{iprx3}, where various sets of sufficient conditions for a map to have  the rate are given.

A {\it branched covering} of the sphere is a covering map of the sphere except at a finite number of critical points were it is locally conjugate to 
$z\to z^k$ in the open disc, for some integer $k\geq 2$ depending on the critical point. The {\it multiplicity} of the critical point is $k-1$  and it is well known (Riemann-Hurwitz formula) that the sum of the multiplicities of critical points is 
equal to $2d-2$, where $d$ is the degree of $f$. Furthermore, if a region $R$ is completely invariant and simply connected, then it contains exactly $d-1$ critical 
points (counted with multiplicities). 

As pointed out before,  having exactly one fixed critical point of 
multiplicity $d-1$ in the boundary of $R$ is an obstruction to having the rate.  In the case that the boundary of $R$ is locally connected, we show that this is the only one:\\

\noindent
{\bf Theorem A.} 
{\em Let $f$ be a degree $d$ branched covering of the sphere, where $|d|>1$. Assume that there exists a completely invariant simply connected region $R$ whose boundary component is locally connected. Assume moreover that it is not the case that there exists only one critical point in the boundary of $R$ that has multiplicity $d-1$ and is fixed by $f$. Then $f$ has the rate.}\\

We do not know if the local connectivity hypothesis is  necessary (but suspect it is not).
A main ingredient in the proof of Theorem A is Theorem \ref{t2}, that states that $f$ extends continuously to the prime end closure of $R$. Now, this extension gives a circle map of 
degree $d$ which, of course, has the rate. Now, using that  the boundary of $R$ is locally connected, to each periodic point in this circle at infinity, corresponds a periodic point of
$f$ in the boundary of $R$. However, this correspondece is not injective, so in order to get the rate one has to understand how many different rays can land at the same point.

An example to have in mind is when $f$ is a complex polynomial with connected and locally connected Julia set. Then $f$ has a supperatracting fixed point at infinity and the region $R$ is
its basin of attraction, which is the complement of the filled Julia set. Figure \ref{rayos} shows different periodic rays landing at the same point, namely the periodic orbit 
$1/7 \to 2/7 \to 4/7$ is reduced to a point in the boundary of $R$.  More figures illustrating this
phenomenon can be found in Chapter 18 in \cite{milnor} .
\begin{figure}[ht]
\psfrag{17}{$\frac{1}{7}$}
\psfrag{27}{$\frac{2}{7}$}
\psfrag{47}{$\frac{4}{7}$}

\begin{center}
{\includegraphics[scale=0.3]{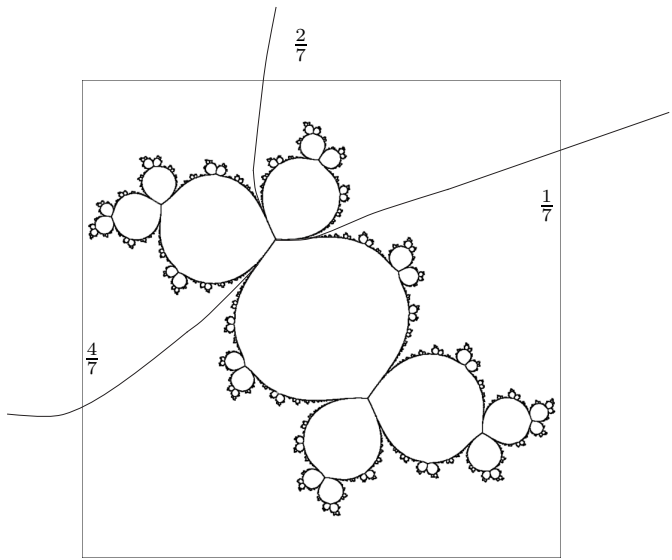}}

\caption{Julia set for $f(z) = z^2 - 0.110 + 0.6557i$}\label{rayos}
\end{center}
\end{figure}

So, in particular, 
Theorem A is a topological version of the fact that complex polynomials (with connected and locally connected Julia set) have the rate.  Following the polynomial analogy, the opposite situation 
corresponds to the case when all  critical points belong to the complement of the filled Julia set. We also give a topological version of the fact
that such maps have the rate: \\

\noindent
{\bf Theorem B.} 
{\em Let $f$ be a degree $d$ branched covering of the sphere, where $|d|>1$. Assume that there exists  a simply connected open set $U$ whose closure is disjoint from the set of critical
values and
such that $\overline{f^{-1} (U)}\subset U$. Then $f$ has the rate.}\\

\section{Proof of the Theorems.}

In this section we prove a theorem extending Caratheodory's theory. This result will be used to establish a relation between the periodic points of $f$ and those of $m_d$, the map 
defined in the unit circle as $m_d(z)=z^d$.


\subsection{Maps induced in the circle.}\label{21}

In this section no dynamics are involved. It contains just a generalization of a well known theorem of extension of a homeomorphism of a simply connected plane
region $R$ to its prime ends closure : Caratheodory's Theorem. There are no assumptions on the structure of the boundary of $R$. We will show that a branched covering also extends to 
the prime ends closure of $R$. We recall
some definitions first. A crosscut in $R$ is a simpe arc whose interior is contained in $R$ and whose extreme points belong to $K$, the boundary of $R$. Fix a point $0$ in $R$. Given 
a crosscut $c$ define $N(c)$ as the component of $R\setminus c$ not containing $0$. A sequence of crosscuts $\{c_n\}$ is admissible if their lenghts tends to zero and
$N(c_{n+1})\subset N(c_n)$. Two sequences $\{c_n\}$ and $\{c'_n\}$ of admissible crosscuts are equivalent if for every $m$ there exist $n$ and $n'$ such that
$N(c'_m)\subset N(c_n)$ and $N(c_m)\subset N(c'_{n'})$. An equivalence class of admissible crosscuts define a prime end. This construction has the following property proved by 
Carath\'eodory: Let $\tilde R$ be the union of $R$ with the set of prime ends. Then there is a topology on $\tilde R$ turning $\tilde R$ into a topological closed disc whose
interior is homeomorphic to 
$R$ with its plane topology. Moreover, if $f$ is a homeomorphism of the closure of $R$ in the plane, then $f$ has a continuous extension to $\tilde R$. The following is an extension of
this result.
 
\begin{thm}
\label{t2}
Let $f$ be a degree $d$ branched covering of the sphere, and $R$ a simply connected completely invariant region whose complement contains more than one point. Then $f$ extends to a continuous map $\tilde f$ in the prime ends closure of $R$. Moreover $\tilde f$ restricted to the boundary of $\tilde R$ is a 
degree $d$ covering map.
\end{thm}

\begin{proof}
As the complement of $R$ has more than one point and is connected, we can assume that $R$ is a bounded region of the plane $\R^2$.
By the assumption that $f$ is a degree $d$ branched covering of $R$, it follows that $f$ has $d-1$ critical points 
(counted with multiplicities) in $R$ (Riemann-Hurwitz formula).
Take $\gamma$ a simple closed curve in $R$ such that all critical values lie in $D$, the bounded component of $\R^2\backslash \gamma$.  Then, $f^{-1}(\gamma)$ is a simple closed curve, and 
$f|_A: A\to f(A)$ is a degree $d$ covering map, where $A$ is the annulus bounded by $\partial R$ and $f^{-1} (\gamma)$.
%
Let $p$ denote a prime end. It is claimed first that there exists a sequence $\{c_n\}$ of crosscuts defining $p$ such that $f$ restricted to $c_n$ is injective for each $n$. First note
that the crosscuts can be constructed in such a way that both extreme points have different images, as the critical points in the boundary of $R$ are also branched points. Let $\delta>0$ be 
the diameter of $D$.  If each $c_n$ is small enough then $f(c_n)$ does not intersect $D$ and has diameter less than $\delta$ (use here that $f$ extends continuously to the closure of $R$). 
Let $x$ and $y$ be different points in $c_n$ such that $f(x)=f(y)$. Then the image under $f$ of the segment $\alpha$ in $c_n$ joining $x$ and $y$ cannot be homotopically trivial in 
$f(A)$. But this is absurd by the choice of $\delta$ and the crosscuts $c_n$. This proves the claim.\\
The claim implies immediately that $\{f(c_n)\}$ is a sequence of crosscuts, and thus defines a prime end $\tilde f(p)$. Continuity is obvious by the definition of the topology of 
$\tilde R$ and the continuity of $f$ in the closure of $R$, so it remains to prove the last assertion. Given a prime end $p$ defined by a sequence of crooscuts $\{c_n\}$, let $\beta$ be
a simple arc in $R$ joining $\gamma$ with a point in the boundary of $R$ and such that $\beta\cap c_n=\emptyset$ for all $n$. Then the preimage of $\beta$ under $f$ is the union of $d$ 
simple arcs each one of which joins $f^{-1}(\gamma)$ with a point in the boundary of $R$ and whose interiors are pairwise disjoint (recall that $f|_A: A\to f(A)$ is a covering). Then 
$R\setminus f^{-1}(\beta)$ has exactly $d$ connected components, restricted to each of which the map $f$ is injective and onto $R\setminus\beta$. It follows that each $c_n$ has an 
$f$-preimage in each of these components, so $\tilde f^{-1}(p)$ contains exactly $d$ points.
\end{proof}

Note that no condition was imposed on the boundary of $R$. Note also that $f$ may have critical points in the boundary of $R$, but in any case it is part of the assertion of the 
theorem that the restriction of $\tilde f$ to the boundary of $\tilde R$ (that is homeomorphic to the circle) has no critical points. This is illustrated in the following example.\\

{\bf Example.}
The complex polynomial $f(z)=z^2-2$ satisfies the hypothesis of the theorem above. The region $R$ is the complement of the interval $[-2,2]$ in the real axis. The critical point $0$ belongs to the boundary of $R$. However, each point in the open interval represent two different prime ends while the extreme points $2$ and $-2$ represent only one prime end. The two prime ends whose impression is the critical point $0$ have the same image under the map $\tilde f$. This is the reason why $\tilde f$ has no critical points.

\subsection{Existence of periodic rays.}

Beginning with the proof of Theorem A, let $f$ be a degree $d$ branched covering of the sphere and let $R$ be a completely invariant simply connected region. 
By assumption, the boundary of $R$ is not a unique point, so Theorem 1 holds for $f$. There is no loss of generality in assuming that there is a unique critical point in $R$, 
say $\infty$,  with multiplicity $d-1$.  Indeed, as in the proof of Theorem 1, $f$ is a covering map of degree $d$ from a semi-closed annulus which is a neighbourhood of
$\partial R$ onto its
image, and we may collapse the boundary of this annulus to a point, which will be critical.  As the periodic points will be found on $\partial R$, where the dynamics of $f$ 
remains unchanged,
the proof under this assumption will suffice. 

A ray in $R$ is a simple arc joining $\infty$ with a point in $K$, the boundary of $R$. For each prime end one can choose a ray intersecting every crosscut defining the prime end.  
As we are assuming that $K$ is locally connected, each prime end defines a point in $K$ (see, for example, Chapter 17 in \cite{milnor}). Moreover, if the prime 
end $p$ is $\tilde f^k$- fixed, so is the landing point for $f^k$. Indeed, in this case, if $c_n$ is a sequence of crosscuts defining $p$, then $f^k(c_n)$ is a sequence of crosscuts 
 defining the  same prime end as $c_n$ (see subsection \ref{21}). 

By Lemma 1 in \cite{iprx1}, $\tilde f|_{S^1}$ is semiconjugate to $m_d(z)=z^d$ acting on $S^1$, this means that there exists a continuous degree one map $h:S^1\to S^1$ such that 
$h\tilde f=m_dh$. Moreover it is shown that $h$ is monotonically increasing, meaning that if $\pi:\R\to S^1$ is the universal covering of the circle, then any lift $H$ of $h$ is 
monotonically increasing. Of course $h$ may have intervals where it is constant, but the fact that it has degree one implies the following: if $x$ and $y$ are different points with
the same $\tilde f$-image, then $h(x)\neq h(y)$ (see item (2) after Definition 2 in the above cited reference).

It is easy to find right inverses of $h$ $(h\phi_0=id)$. Of course none of them will be continuous, unless the semiconjugacy $h$ is actually a conjugacy. Choose a right 
inverse $\phi_0$ of $h$ such that $\phi_0m_d=\tilde f\phi_0$ and $\phi_0$ is monotonically increasing.

Now the assumption that $K$ is locally connected will be used to define a map $I:\partial\tilde R\to K$ where $I(p)$ is the impression of $p$, a unique point in $K$. Note that $I$ is continuous, surjective and $fI=I\tilde f$. The map $I$ is not injective since different prime ends may have the same impression in $K$. 

Of course, if two rays $r_1$ and $r_2$ representing different prime ends $x_1$ and $x_2$, land at the same point $y\in K$, then this point separates $K$. Moreover, 
the union of these rays with $y$ separates the whole sphere, and  $I$ sends each component of $S^1\setminus \{x_1,x_2\}$ into the
closure of a component of $K\setminus \{y\}$. It may happen as well, that some point in the interior of an arc from $x_1$ to $x_2$ also has its impression 
at the point $y$ (see Figure 1). 

\begin{figure}[ht]
\psfrag{phi0}{$\phi_0$}
\psfrag{h}{$h$}
\psfrag{i}{$I$}
\psfrag{r}{$R$}
\psfrag{r1}{$r_1$}
\psfrag{r2}{$r_2$}
\psfrag{r3}{$r_3$}
\psfrag{rr}{$\widetilde{R}$}

\psfrag{ff}{$\widetilde{f}$}
\psfrag{md}{$m_d$}

\psfrag{14}{\tiny{$\frac{1}{4}$}}

\psfrag{11}{$\frac{1}{16}$}
\psfrag{4}{\tiny{$4$}}
\begin{center}
{\includegraphics[scale=0.21]{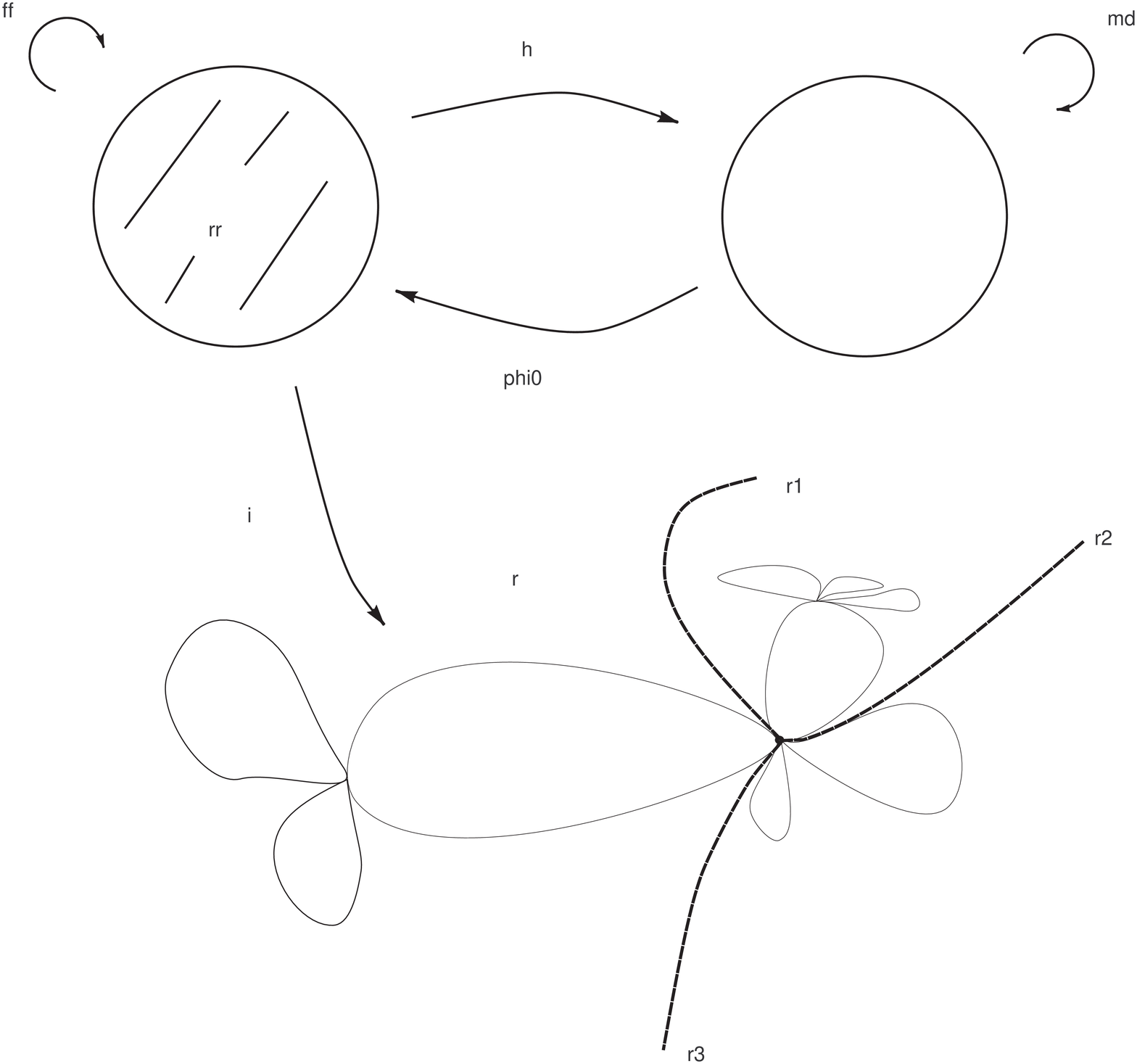}}

\caption{}
\end{center}
\end{figure}

We turn now into the ideas for the proof of Theorem A. Note that $\tilde f$, being a degree $d$ covering map of the circle, has at least $d^n-1$ points with period that divides $n$,
or which is the same, $\tilde f^n$ has at least $d^n-1$ fixed points: this is obvious for $m_d$ and follows for $\tilde f$ because of the semiconjugacy. As was explained earlier, 
to each of the fixed points of $\tilde f^n$ corresponds (taking impressions) a fixed point of $f^n$. However, as many different points may have the same impression, in order to have the rate
one has to control this possibility. Counting will be easier to perform with the map $m_d$ instead of $\tilde f$, so define the map $\phi=I\phi_0$, that satisfies $\phi m_d=f\phi$, and so it carries periodic points of $m_d$ into periodic points of $f$.

\begin{defi}
Two $m_d$-periodic points $x$ and $y$ in $S^1$ are said equivalent if $\phi(x)=\phi(y)$. This will be denoted as $x\sim y$.
\end{defi}

Of course this is an equivalence relation.  Note that $w_1\neq w_2$ and $\phi(w_1)=\phi(w_2)$, then $K\backslash \{\phi(w_1)\}$ is disconnected.

%

\begin{lemma}
\label{l1}
Let $x$ and $y$ be different points in the boundary of a component of the complement of the closure of $R$ such that $f(x)=f(y)$. 
Then there exist points $x'$ and $y'$ in $S^1$ having the same image under $m_d$, such that $x'\in h(I^{-1}(x))$, $y'\in h(I^{-1}(y))$ and the following property holds:\\
$\phi(w_1)=\phi(w_2)$ implies that either $w_1$ and $w_2$ belong to the same component of $S^1\setminus\{x',y'\}$, or 
$\phi(w_1)=\phi(w_2)=\phi(x')$ or $\phi(w_1)=\phi(w_2)=\phi(y')$.\\  
\end{lemma}
\begin{proof}
Let $r$ be a ray landing at $z=f(x)$ and $r_x$, $r_y$ rays landing at $x$ and $y$ respectively such that $f(r_x)=f(r_y)=r$. Note that $r_x$ and $r_y$ are different as points in $\partial\tilde R$, and will be denoted by $x_0$ and $y_0$ as points in $S^1$. Note that $\tilde f(x_0)=\tilde f(y_0)$, which implies, as was explained above, that $h(x_0)\neq h(y_0)$. Let $x'=h(x_0)$ and $y'=h(y_0)$.

The assumptions on $x$ and $y$ imply that there exists a simple arc $s$ joining $x$ and $y$ in the sphere with the property that the interior of $s$ does not intersect the closure of $R$.
It follows that $S^2\setminus\{s\cup r_x\cup r_y\}$ has exactly two connected components, and that in both components there are points of $K$. Note that as $h$ and $\phi _0$ are monotonic, if $w_1$ and $w_2$ are points in different components of $S^1\setminus \{x',y'\}$, then $\phi_0(w_1)$ and $\phi_0(w_2)$ belong to different components of $S^1\setminus \{x_0,y_0\}$. For $i=1,2$, let $r_i$ be rays corresponding to $\phi_0(w_i)$ not intersecting $r_x$ nor $r_y$ within $R$. The interior of the rays $r_1$ and $r_2$ must belong belong to different components of $S^2\setminus (s\cup r_x\cup r_y)$. As $\phi(w_1)=\phi(w_2)$ implies that $r_1$ and $r_2$ land at the same point in $K$, this point is necessarily $x$ or $y$.
\end{proof}

\section{Proof of Theorem A.}

The fundamental idea is the following: the map $m_d$ acting on $S^1$ has the rate: indeed, $m_d^n$ has $d^n-1$ fixed points. Moreover, the image under $\phi$ of a $m_d$-periodic point
is $f$-periodic. The lemma proved above shows that the points $x'$ and $y'$ obtained separate the circle in such a way that (almost always) a point in one component cannot be identified 
by $\phi$ with a point in another component. As $x'$ and $y'$ have the same image under $m_d$, we will construct in an abstract setting maximal sets of points separating the circle 
in pieces such that points in different pieces cannot be identified. The next subsection is devoted to this. 

\subsection{Stars.}

The procedure begins with some abstract definitions and properties; in the next subsection the construction is realized for the map $f$.
Throughout the following, the circle is considered with the distance $dist$: arc lenght divided by $2\pi$. So the circle has length equal to $1$. Then two different points having the same image under $m_d$ are at a distance $j/d$ for some integer $0<j<d$. 

\begin{defi}
Let $d$ be an integer greater than one.
A $d$-star (or simply a star when no confusion can arise) is a subset $E$ of $S^1$ containing at least two points and such that the distance between any two points in $E$ is equal to $j/d$ for some integer $j$, $1\leq j \leq d-1$. The multiplicity of a star is $m(E)=n-1>0$ if $E$ has $n$ points. Two stars $E_1$ and $E_2$ are disjoint if $E_2$ is contained in the closure of a component of the complement of $E_1$ (which obviously implies that $E_1$ is as well contained in the closure of a component of the complement of $E_2$). In other words, $E_1$ and $E_2$ are disjoint if at most one component of $S^1\setminus E_1$ intersects $E_2$.
A cycle of stars is a sequence of pairwise disjoint stars $\{E_1,\ldots,E_k\}$ such that there exists points $x_i : 1\leq i\leq k$ and $x_i\in E_i\cap E_{i+1}$ for $i<k$ and 
$x_k\in E_k\cap E_1$. 
A set ${\mathcal E}=\{E_1,\ldots,E_k\}$ is a maximal set of $d$-stars if every $E_i$ is a $d$-star, the $E_i$ are pairwise disjoint, there are no cycles in ${\mathcal E}$ and it is maximal respect to these properties.
\end{defi}

For example, if $d=4$, $E_1= \{1,i\}, E_2= \{i,-1\},E_3= \{-1,-i\}, E_4=\{1,-1\}, E_5=\{\-i,1\}$, then $\{E_1, E_2, E_3\}$ is a maximal set of stars,  $\{E_1, E_2, E_3, E_5\}$ is 
a cycle of stars and $\{E_3, E_4\}$ is a set of 
disjoint 
stars which is not maximal (Figure 1, (a),(b) and (c) respectively). We have also drawn a maximal set of stars as well as a cycle of stars for $d=6$ in Figure 2.
 
\begin{figure}[ht]

\begin{center}
\subfigure[]{\includegraphics[scale=0.21]{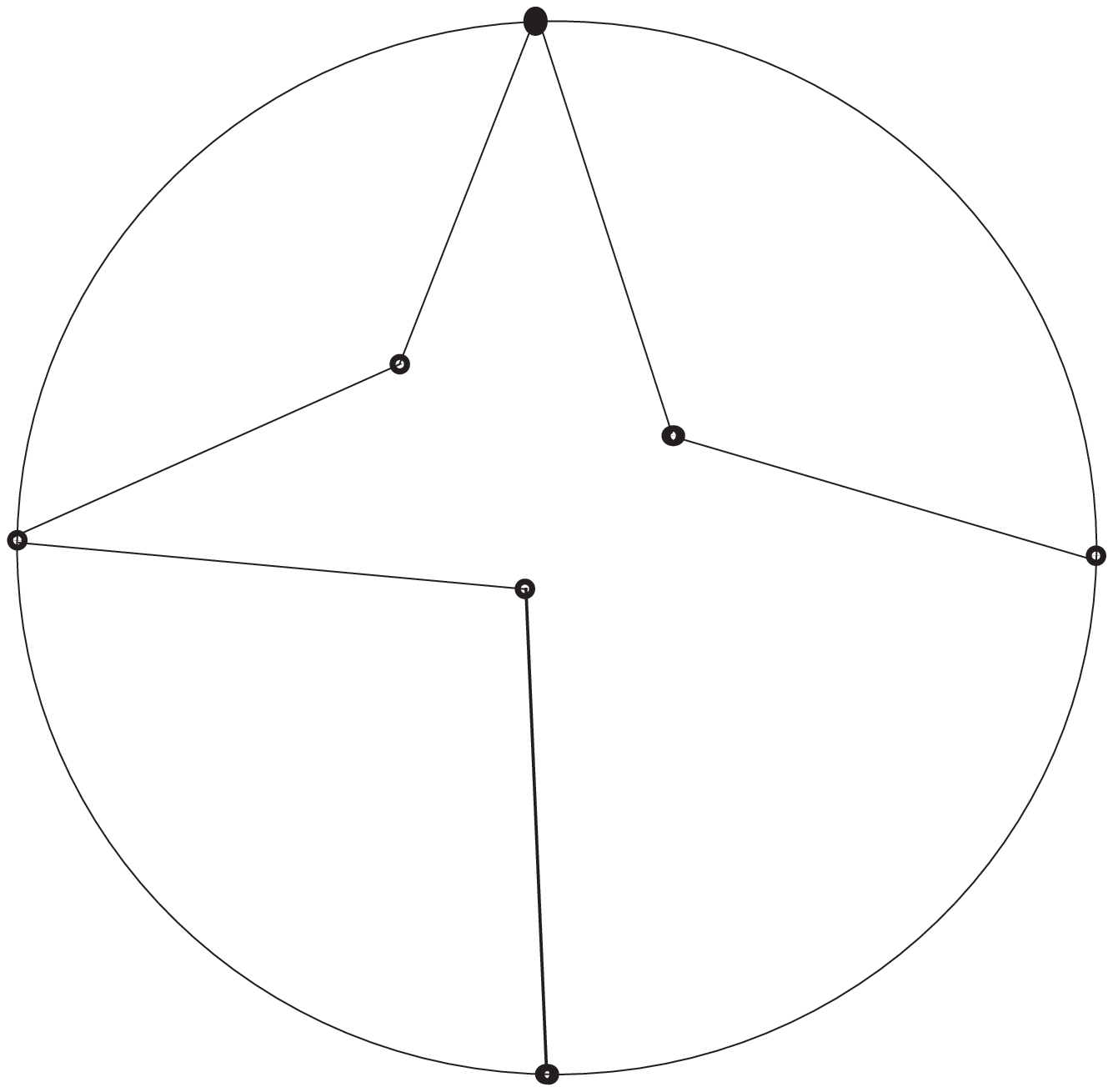}}
\subfigure[]{\includegraphics[scale=0.21]{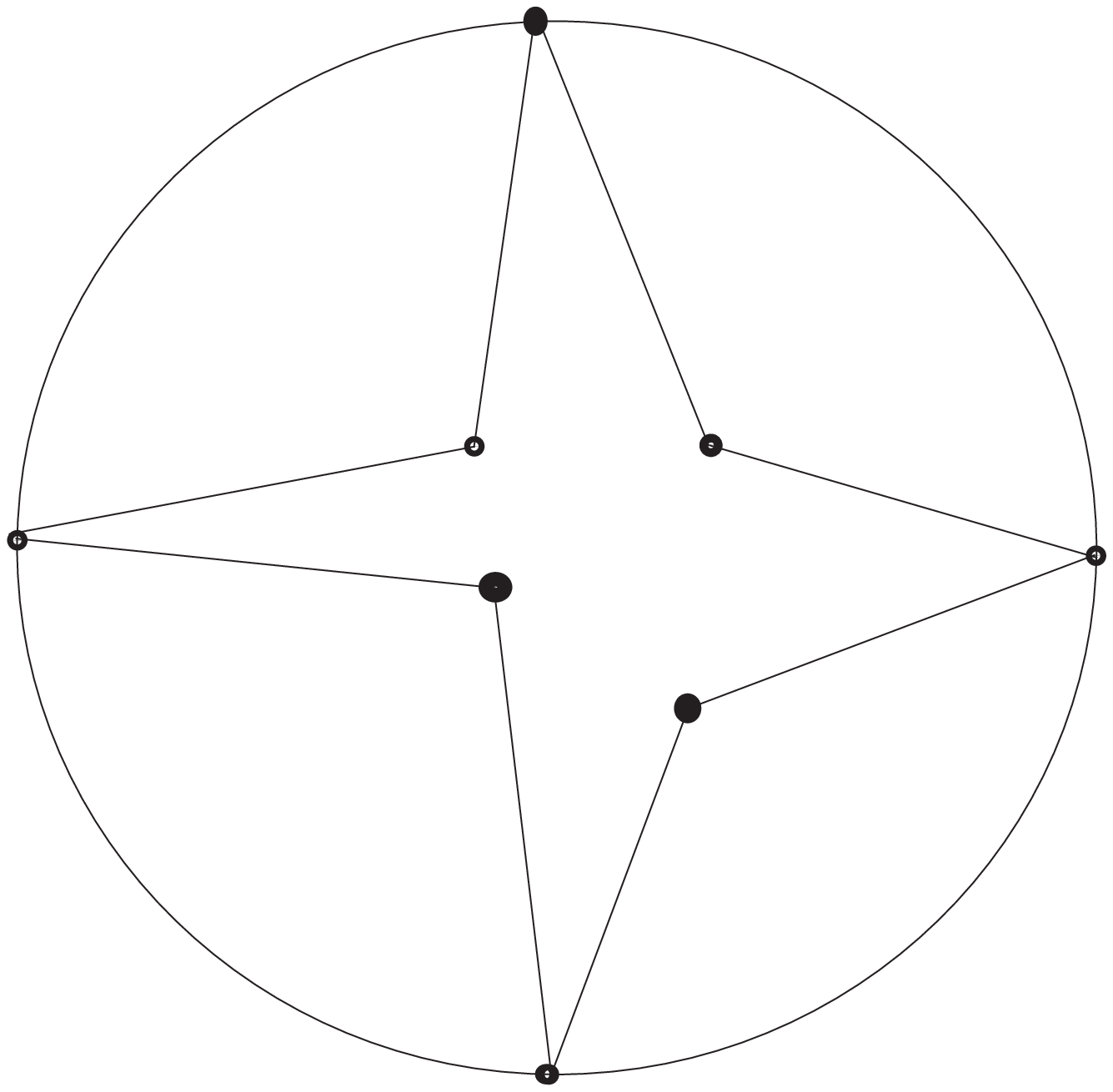}}
\subfigure[]{\includegraphics[scale=0.21]{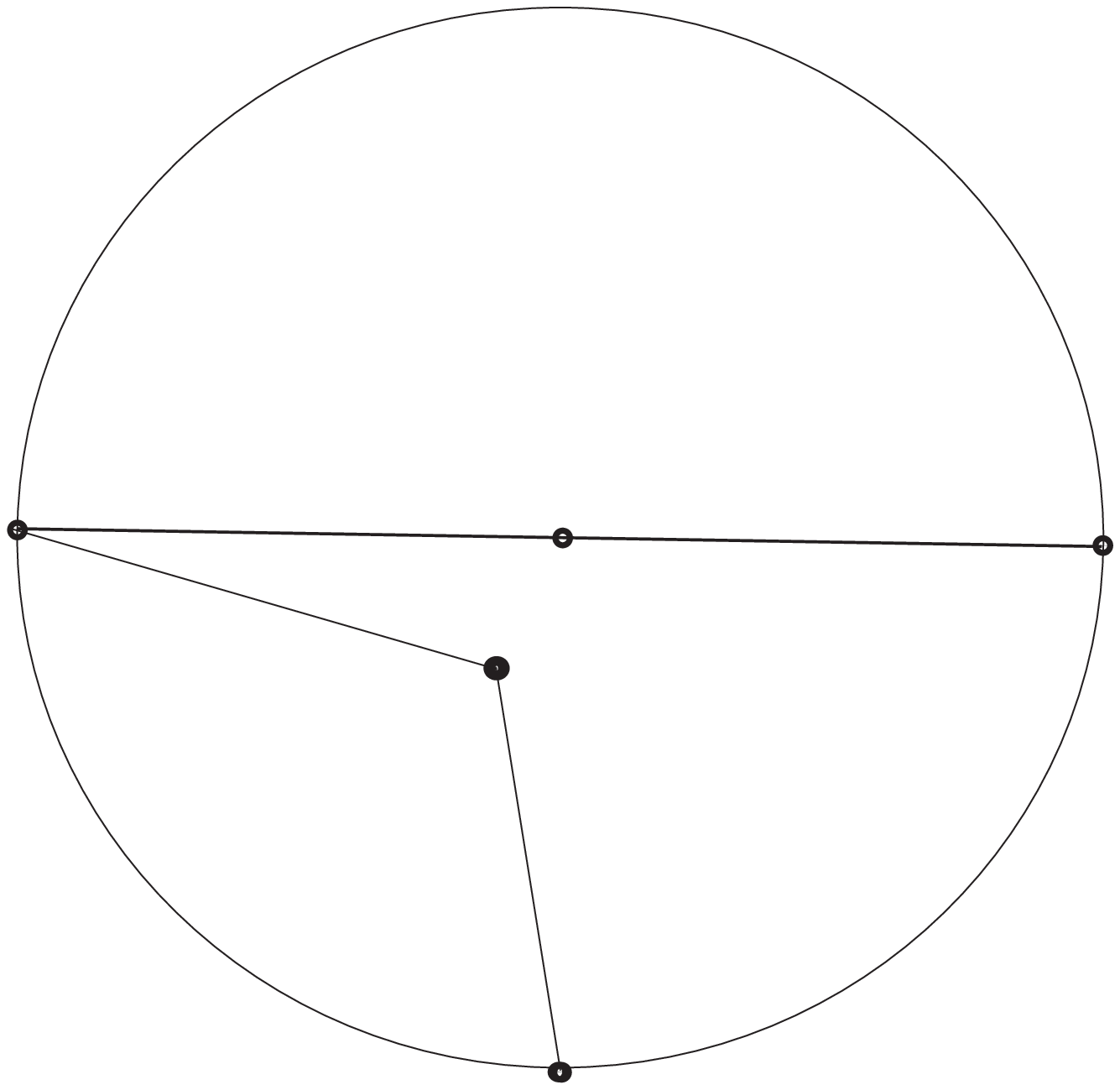}}

\caption{Some stars for $d=4$}\label{figura1}
\end{center}
\end{figure}

\begin{figure}[ht]

\begin{center}
\subfigure[]{\includegraphics[scale=0.21]{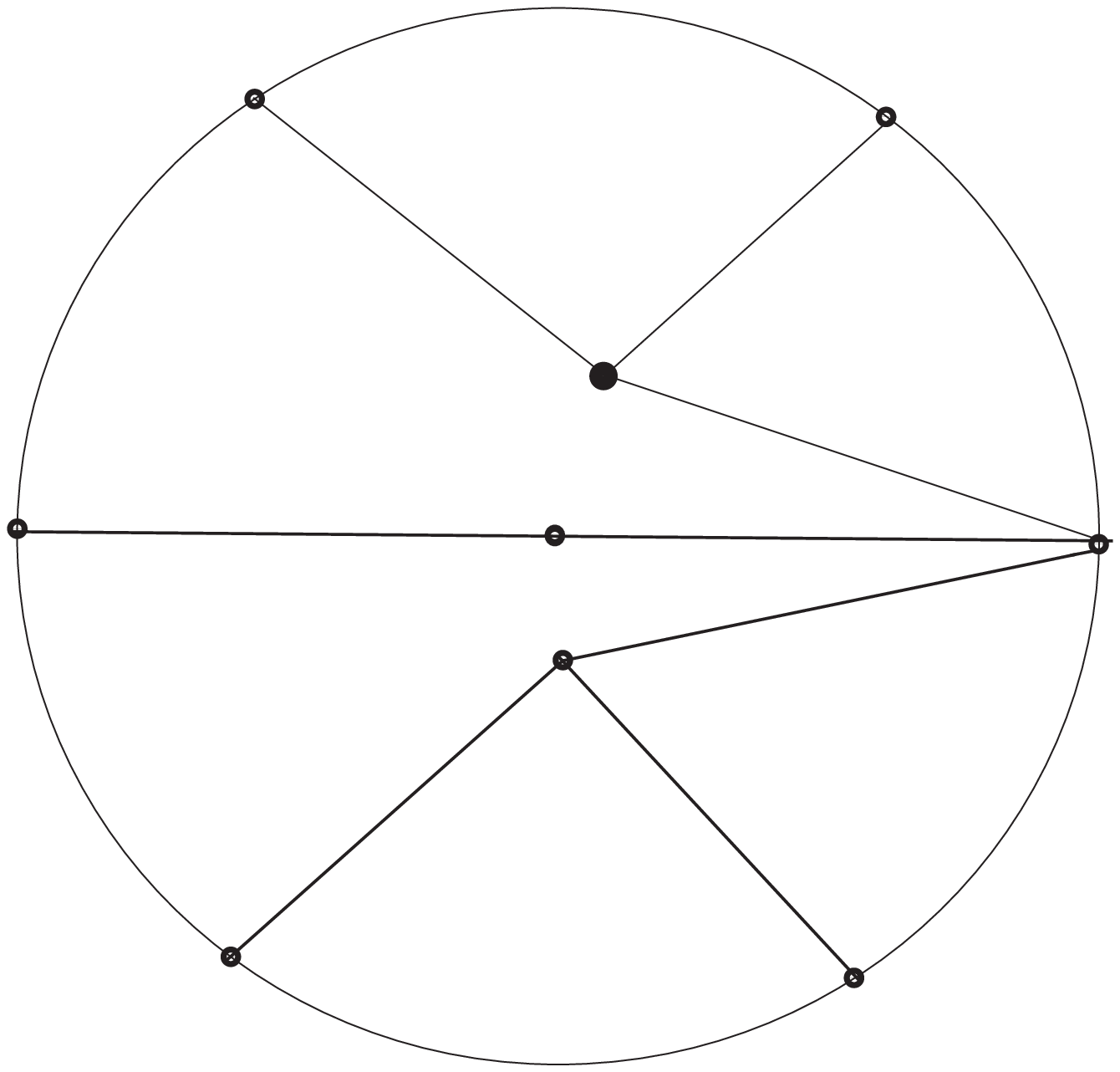}}
\subfigure[]{\includegraphics[scale=0.21]{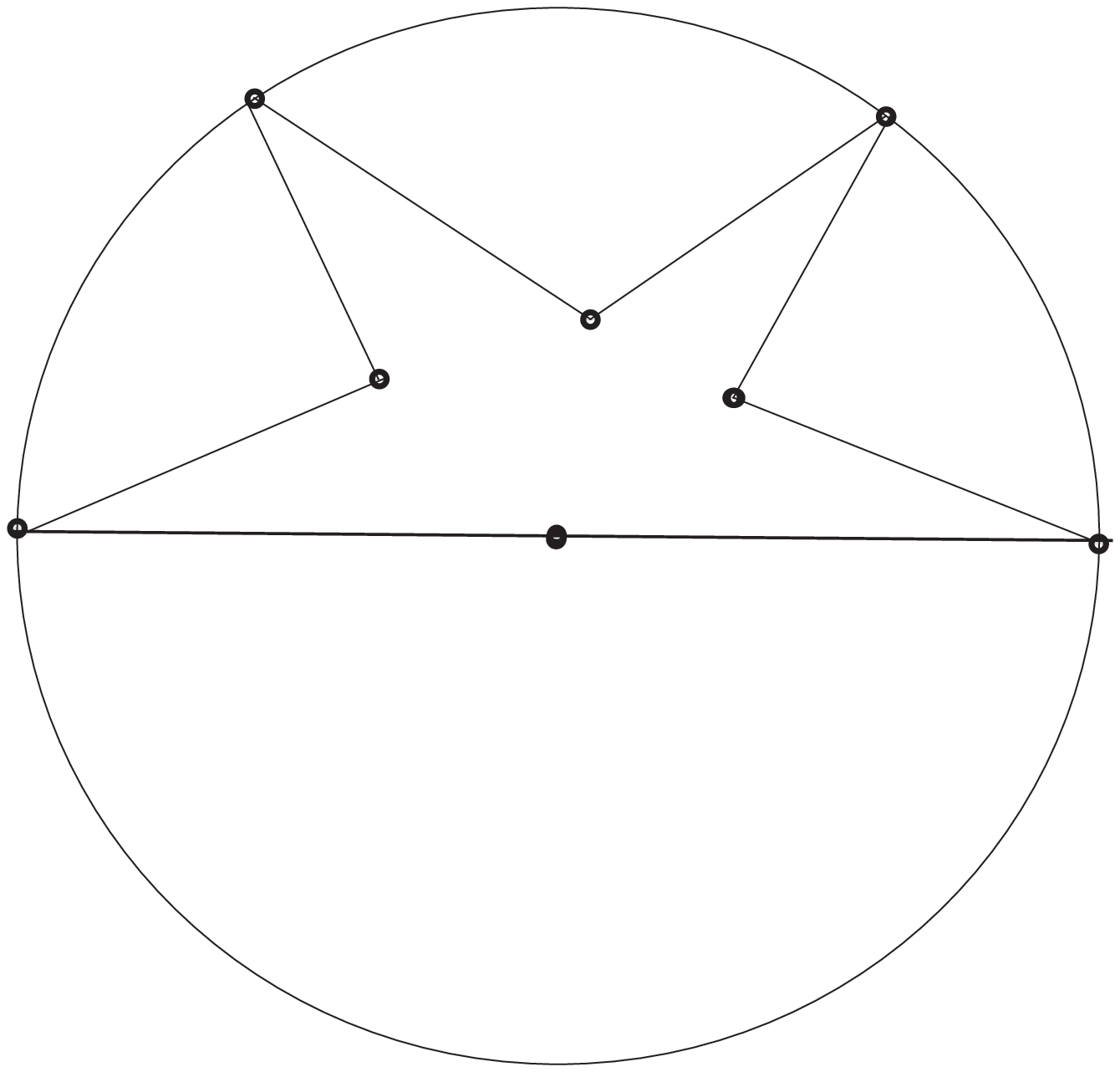}}

\caption{Some stars for $d=6$}\label{figura2}
\end{center}
\end{figure}

In our context each critical point will give rise to a star and the set of stars so defined will be maximal.

\subsection{Properties of stars.}

The main result of this subsection is:

\begin{prop}
\label{p2} Assume that $\mathcal E= \{E_j\}_{j=1}^k$ is a set of disjoint $d$-stars with no cycles.\\
Then $\mathcal E$ is maximal if and only if the sum of the multiplicities of the $E_i$ is equal to $d-1$.
\end{prop}
\begin{proof}
The proof will be done by induction on $d$. The case $d=2$ is trivial. Also the case where there is only one star in the set.
To reduce $d$ and use the induction hypothesis we will use the following. Assume that $\mathcal E=\{E_1,\ldots,E_k\}$ is a set of pairwise disjoint $d$-stars without cycles. Let $x_1$
and $x_2$ be consecutive points in $E_1$, meaning that there is an arc $s$ from $x_1$ to $x_2$ which does not intersect $E_1\setminus\{x_1,x_2\}$. A circle $\tilde s$ is obtained 
from $s$ when the points $x_1$ and $x_2$ are identified. As in the definition of stars the circle has length equal to one, then the quotient $\tilde s$ has to be rescaled with the 
constant $\ell/d$, (where $\ell/d=dist(x_1,x_2)$) to obtain a circle with length equal to $1$. So a $d$-star $E$ in $\mathcal E$ contained in $S^1$ can be also considered as
a $\ell$-star $E'$ contained in $\tilde s$. Having these considerations in mind define $\mathcal E'$ as the set of $E'_i$ such that $E_i$ belongs to $\mathcal E$ and is contained in 
the arc $s$. Then $\mathcal E'$ is a set of disjoint $\ell$-stars without cycles.

\begin{lemma}
\label{cociente}
If $\mathcal E$ is maximal then $\mathcal E'$ is a maximal.
\end{lemma}
\begin{proof}
To prove maximality of $\mathcal E'$, define for $i=1,2$ numbers $\delta_i$, to be the distance between $x_i$ and the union of $E'_1\cup\cdots\cup E'_r$. Note that the
sum $\delta_1+\delta_2$ is less than or equal to $1/d$, otherwise the set $\mathcal E$ was not maximal: indeed, if that sum is greater than $1/d$ then one can add a new star 
(having two points) to $\mathcal E$ to obtain a new set of disjoint $d$-stars without cycles and this contradicts the maximality of $\mathcal E$; and if it is equal to $1/d$ 
then one can add a disjoint new disjoint star in $\mathcal E'$, but losing the no cycles property. Therefore, the set $\mathcal E'$ in $\tilde s$ is maximal.
\end{proof}

Following with the proof of the Proposition, note that the same can be done in each connected component of $S^1\setminus E_1$, where $E_1=\{x_1,\ldots,x_k\}$ is cyclically ordered. The arc $s_i$ is $(x_i,x_{i+1})$ ($i$ is taken to range over $1,\ldots,k$ where $x_{k+1}=x_1$). Each arc induces a circle $\tilde s_i$ and the set of stars 
in $\mathcal E$ that are contained in $s_i$ is denoted $\mathcal E'_i$. By the lemma each $\mathcal E'_i$ is maximal if $\mathcal E$ is maximal. The converse also holds, and is trivial since the stars are disjoint.

Assume the assertion of the proposition true for every $\ell<d$, this means that each $\tilde E'_i$ is maximal if and only if the sum of its multiplicities is equal to $d. dist (x_i,x_{i+1})-1$. Next note that the multiplicity of $E_1$ is $k-1$, so every $\mathcal E'_i$ is maximal if and only the sum of the multiplicities of all the stars is equal to
$$
(k-1)+\sum_1^k (d.dist(x_i,x_{i+1})-1)=(k-1)+(d-k),
$$
since the sum of the the distances $dist(x_i,x_{i+1})$ equals to one.
\end{proof}

\subsection{Construction of stars.}

Let $c$ be a critical point of multiplicity $j$ contained in $K$; take a ray $r$ landing at $f(c)$ and note that there are exactly $j+1$ rays preimage of $r$ landing at $c$. The 
images under $h$ of the prime ends corresponding to these $j+1$ rays is a $d$-star.\\
Let $R_1$ be a component of the complement of the closure of $R$ containing $k$ critical points counted with multiplicities. Choose any point $z$ that is neither periodic nor a 
critical value, in the boundary of $f(R_1)$ and take a ray $r$ in $R$ landing at $z$. Then $f^{-1}(r)$ contains $k+1$ rays landing at 
different points of the boundary of $R_1$, denoted 
$r_1,\ldots,r_{k+1}$. 
Let $E_1=\{h(r_1),\ldots,h(r_{k+1})\}$ (here we are identifying rays and their corresponding prime ends). As $\tilde f(r_i)=\tilde f(r_j)$ it follows that $h(r_i)\neq h(r_j)$ 
whenever $j\neq i$ (see item (2) after 
Definition 2 in \cite{iprx1}) . 
Then $E_1$ has $k+1$ different points, and all of them have the same image under $m_d$: it follows that $E_1$ is a $d$- star of multiplicity $k$.

%
So, with this proceeding, we obtain a star for each critical point in $K$ and for each component of the complement of $\overline R$ containing critical points.\\

\begin{lemma}\label{max}
The union of the stars constructed above is a maximal set.
\end{lemma}
\begin{proof}
Recall that there are exactly $d-1$ critical points in the complement of $R$ counted with multiplicities. It follows that the sum of the multiplicities of the stars constructed is equal to $d-1$, so in view of Proposition \ref{p2} it suffices to show that the stars are disjoint and have no cycles.

That the stars are disjoint follows by construction and because $h$ is monotonic.
 
Assume by contradiction that there exists a cycle $E_1,\ldots E_k$. Taking a minimal cycle (one that does not properly contain another cycle) it can be assumed that the points 
$x_i$ ($1\leq i\leq k$) giving the cycle are not repeated. It is claimed first that the stars are cyclically ordered. Let $E_1$ and $E_2$ have the point $x_1$ in common. This  
determines two  arcs in $S^1$: $a_1$ which contains all the points in $E_1$, is an arc starting at $x_1$ not containing any point in $E_2$ and ending in the last point of $E_1$, and $a_2$, 
which contains all the points in $E_2$ is an arc starting at $x_1$ not containing any point in $E_1$ and ending in the last point of $E_2$,. These intervals intersect only at $x_1$
except if $E_1, E_2$ is already a cycle, in which case the claim is obvious.
Assume that the star $E_3$, having the point $x_2$ in common with $E_2$, has a point in $a_2$. In this case $E_3$ must be contained in $a_2$, and so the point $x_3\neq x_2$ is also contained in $a_2$. This implies that the subsequent $E_j$ (if any) are all contained in $a_2$, which forces $x_k\in a_2$, a contradiction since $x_k\in a_1$. This implies that the whole $E_3\setminus \{x_2\}$ is contained in the complement of the union of $a_1$ and $a_2$. By a simple induction argument the claim follows.

The assumption that $x_i$ belongs to $E_i$ and to $E_{i+1}$ implies that there are two different rays giving the same image under $h$. As we are assuming that there is a cycle, 
we have two sequences of rays $r_i$ and $s_i$, $1\leq i\leq k$, having the following properties:
\begin{enumerate}
\item
All the rays $r_i$ and $s_i$ are different. 
\item
The rays are in cyclic order: once $s_1$, $r_1$ and $s_2$ are given, they determine an orientation of $S^1$ such that $s_1<r_1<s_2$. With this orientation fixed, the claim above implies that $r_i<s_{i+1}<r_{i+1}$ for every $i$.
\item 
For each $1 \leq i\leq k-1$, the image under $h$ of the oriented interval $I_i=(r_i,s_{i+1})$ is a point, since the extreme points have the same image under $h$ and $h$ is monotone. The same assertion holds for the interval $I_k=(r_k,s_1)$.
\item
It comes from the construction of stars that for each value of $i$, the image under $\tilde f$ of $r_i$ and $s_i$ is the same.
\end{enumerate}

(See Figure \ref{lem3})
Now use that $ \tilde f$ is a covering of $S^1$, together with properties 2) and 4) to deduce that $\tilde f(I)=S^1$, where 
$I=\cup_{i=1}^kI_i$. On the other hand, since by property 3) the image under $h$ of these union is finite, the equality 
$h\tilde f(I)=m_dh(I)$ is contradicted.

\end{proof}

This maximal set of stars will be denoted ${\mathcal E}$.

\begin{figure}[ht]
\psfrag{phi0}{$\phi_0$}
\psfrag{h}{$h$}
\psfrag{i}{$I$}
\psfrag{r}{$R$}
\psfrag{r1}{$r_1$}
\psfrag{r2}{$r_2$}
\psfrag{r3}{$r_3$}
\psfrag{rr}{$\widetilde{R}$}
\psfrag{s2}{$s_2$}
\psfrag{ff}{$\widetilde{f}$}
\psfrag{s1}{$s_1$}
\psfrag{h}{$h$}
\psfrag{i1}{$I_1$}
\psfrag{i2}{$I_2$}

\begin{center}
\subfigure{\includegraphics[scale=0.2]{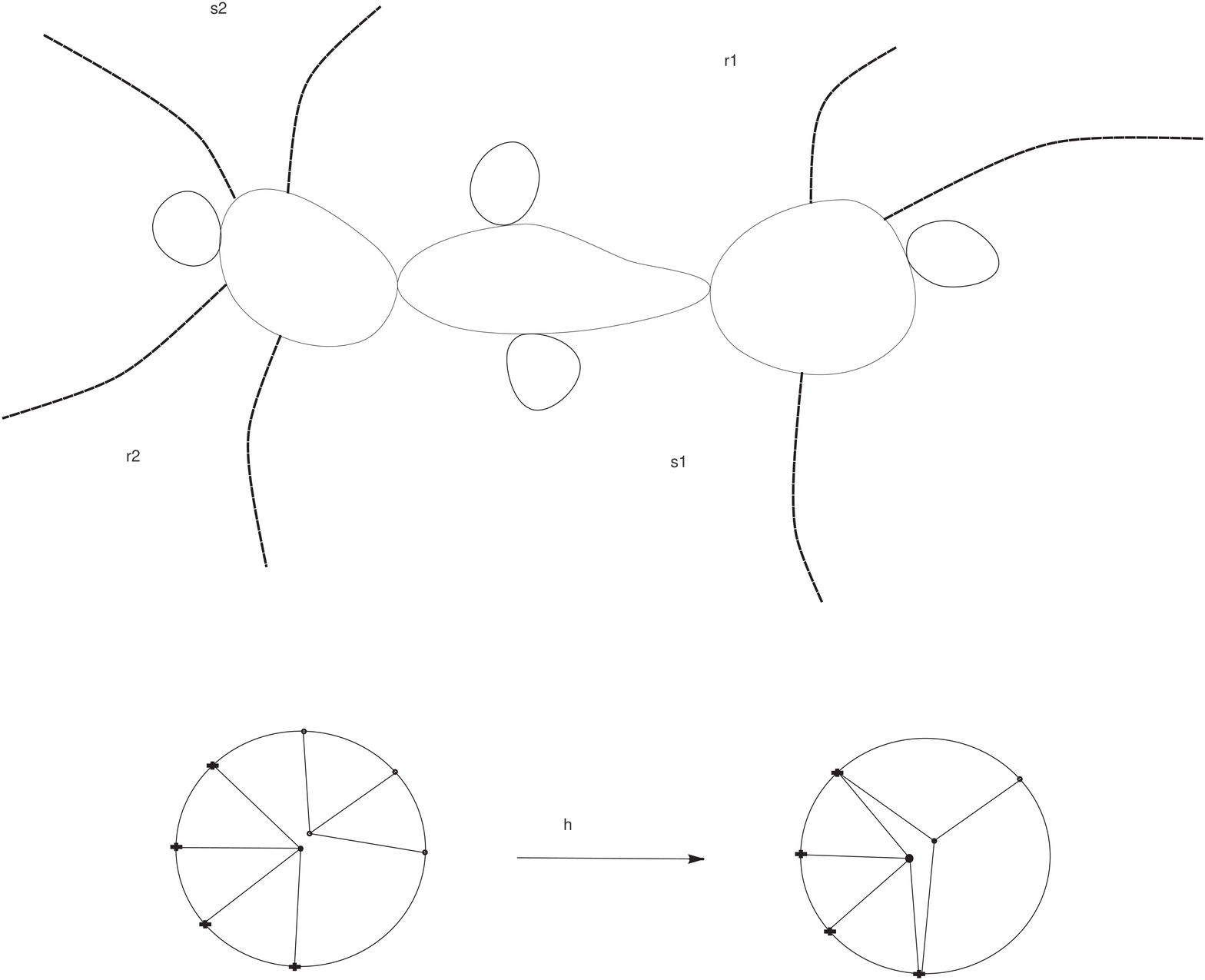}}
\subfigure{\includegraphics[scale=0.21]{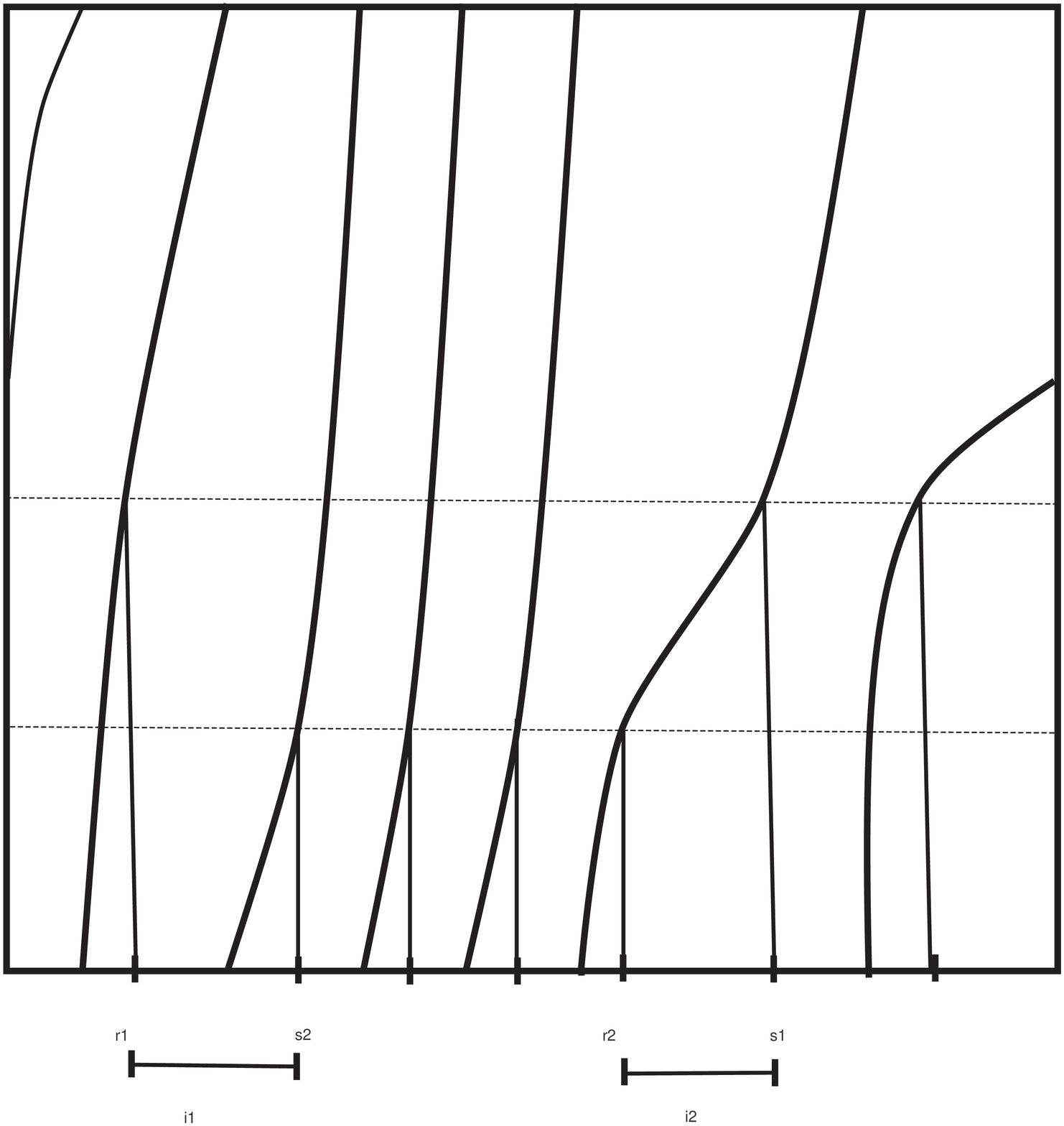}}

\caption{Proof of Lemma \ref{max}}\label{lem3}
\end{center}
\end{figure}

\begin{rk}
\label{r1}
Let $c$ be a critical point in $K$ with multiplicity $k-1$. Then $c$ separates $K$ (recall that $K=\{c\}$ was implicitly excluded in the hypothesis) and the construction above gives rays $s_1,\ldots,s_k$ that land at $c$ and separate the sphere in $k$ components. 
In general, if $c_1$ is another critical point in the complement of $R$ (not necessarily in $K$) with multiplicity $j-1$, then associated to $c_1$ we have defined several rays $r_1\ldots,r_{j}$ that separate $R$ into $j$ components.
Then there are components $W_1$ of $R\setminus\{s_1,\ldots,s_k\}$ and  
$W_2$ of $R\setminus\{r_1,\ldots,r_{j}\}$ with disjoint closures. These components correspond to disjoint connected components $V_1$ and $V_2$ of different stars.
\end{rk}

Take $z\in S^1$ such that $z$ and $m_d^{-1}(z)$ do not intersect stars and such that its forward $m_d$-orbit is dense in $S^1$.
Let $\epsilon$ be a small positive number such that if $I$ is an interval of length $\epsilon$ having $z$ as an extreme point then $m_d^{-1}(I)$ does not intersect any star. Let
$x=\{x_1, \ldots, x_\nu\}$ and $y= \{y_1, \ldots, y_\nu\}$ be periodic orbits of period $\nu$. The following lemma will eventually give us the rate.

\begin{lemma} Let $\nu$ be prime and larger than any period of a critical point. Assume also that 
$\frac{1}{d(d^{\nu}-1)} < d (y_{\nu}, {\mathcal E})$.
If $x_1$ and $y_1$ belong to $I$ and $|x_1-y_1|=1/(d^{\nu}-1)$ then $\phi(x_1)\neq\phi(y_1)$.
\end{lemma}

\begin{proof} As $|x_1-y_1|=1/(d^{\nu}-1)$, there exists an integer $\ell$ such that:
$$
|x_\nu-y_\nu|=\frac{1}{d(d^\nu-1)}+\frac{\ell}{d}.
$$
But $\ell$ cannot be $0$ since the $x_\nu$ and $y_\nu$ are fixed points of $m_d^\nu$ and the distance between such points cannot be less than 
$1/d^\nu-1$.
Note also that by the choice of $\nu$, $x_\nu$ and $y_\nu$ must be separated by a star: otherwise, taking $y'_{\nu} = y_{\nu} -1/(d^{\nu}-1) $ we have that 
$\{x_{\nu}, y'_{\nu}\}$ is a star, disjoint from every other star in 
${\mathcal E}$ contradicting maximality.

As $\phi(x_\nu)=\phi(y_\nu)$ and $x_\nu$ and $y_\nu$ are separated by stars, the unique possibility given by Lemma \ref{l1} is that $\phi(x_\nu)=\phi(y_\nu)=c$ for some critical 
point $c$ as by construction, the landing points of rays defining the stars are not periodic unless they are critical. (The reader should not that the proof stops here if there are no critical
points in the boundary of $R$).

As the period $\nu$ was taken large, the critical point $c$ is fixed. Now, the hypothesis on the critical points implies that $f$ must have another critical point $c'$ in $R^c$. As 
we have two stars, and the forward orbit of $z$ is dense, we may assume that $\phi(z)$ was taken in the component $W_1$  and that $f ^n (\phi(z))\in W_2$, where $W_1$ and $W_2$ are as in 
Remark \ref{r1} ($W_1$ and $W_2$ have
disjoint closures).   If $I$ is small enough, then $m_d^n (I)$ will not intersect any star and therefore $f^n (\phi(I))\subset W_2$.

Now note that $\phi(x_1)$ is fixed by $ f$  but this contradicts the fact that $f^n (\phi(I))\subset W_2$.

 \end{proof}

Note that there are approximately $\epsilon d^\nu$ fixed points of $m_d^\nu$ in the interval $I$ (where $\epsilon$ is the length of $I$) and these points have the
property that no pair of consecutive points are equivalent. There is another restriction:

\begin{lemma}
If $x,y,z,t$ are points in $S^1$ such that $x\sim y$, $z \sim t$ but $x$ and $z$ are not equivalent, then $z$ and $t$ belong to the same component of $S^1\setminus \{x,y\}$.
\end{lemma}
Let $r_x$ and $r_y$ be rays landing at the same point $p$.  Then, $r_x\cup r_y\cup p$ separate the plane. Now if $z$ and $t$ belong to different connected components of $S^1\setminus \{x,y\}$ this means, as they are equivalent, that the landing point of $r_z$ and $r_t$ is also $p$, contradicting that $x$ and $z$ are not equivalent.

Notice that there are $[\epsilon d^\nu]$ fixed points of $m_d^\nu$ in the interval $I$, where $[x]$ is the integer part of $x$.
Let $n=[\epsilon d^\nu]$. The next abstract result implies that there are many different equivalence classes.

\begin{lemma}
\label{saltos}
Let $R$ be an equivalence relation in the set $\{1,\ldots,n\}$ such that the following properties hold:\\
1. $(i,i+1)\notin R$ for every $i$, and \\
2. If $L$ and $L'$ are different classes, then each one of them is contained in a connected component of the complement of the other.\\
Then the number of classes greater than or equal to $[n/2]+1$.
\end{lemma}
\begin{proof}
Assume that the property holds for every number less than $n$ and let $R$ be an equivalence relation in $\{1,\ldots,n\}$. Let $C_1$ be the class of $1$. Denote by $\sigma_1,\ldots,\sigma_k$ the maximal intervals of the complement of $C_1$. Note that $k=n_1$ or $k=n_1-1$, where $n_1$ is the number of elements in $C_1$. If $p_i$ denotes the cardinal of $\sigma_i$, then $n_1+\sum p_i=n$. Say that $i\in E$ if $p_i$ is even and that $i\in O$ if $p_i$ is odd. Let $e$ denote the number of elements in $E$ and $o$ the number of elements in $O$.\\
By the hypothesis (2) on the classes, note that the number $N_R$ of equivalence classes of $R$ is at least
$$
N_1=1+\sum_{i\in E}(1+\frac{p_i}{2})+ \sum_{i\in O}(1+\frac{p_i-1}{2}),
$$
where the first $1$ comes for the class $C_1$, and the induction hypothesis was used in each $\sigma_i$. Rearranging terms it comes that $N_R$ satisfies the thesis of the lemma if and only if
$$
2e+o-n_1\geq 2[n/2]-n.
$$

Note that $e+o = n_1$ or $e+o = n_1-1$ , so $2e+o-n_1$ is either $e$ or $e-1$.  So the last equation is valid if $e>0$ or $e+o=n_1$ since the number on the right hand side is equal to $0$ or $-1$. If $e=0$ and $e+o=o=n_1-1$ then 
$n=\sum_op_i+n_1$ is an odd number, so both sides are equal to $-1$.
\end{proof}

Note: For every $n$ there is an equivalence relation $R$ such that the number of classes is equal to $[n/2]+1$: one class is the set of odd numbers, and every even number constitutes a class.

From this last lemma we obtain Theorem A. Indeed, for every prime number $\nu$ large enough the number of equivalence classes of fixed points 
of $m_d^\nu$ within the interval $I$ is at least $[\epsilon d^\nu/2]$ and as different equivalence classes correspond to different fixed points of $f^\nu$, then the rate of $f$ is not less than 
$$
\lim_\nu\frac{1}{\nu}\log[\epsilon d^\nu/2]=\log d.
$$

\section{Proof of Theorem B}

This section is devoted to the proof of :\\
\noindent
{\bf Theorem B.} {\em Let $f$ be a degree $d$ branched covering of the sphere, where $|d|>1$. Assume that there exists  a simply connected open set $U$ whose closure is disjoint from the
set of critical
values and 
such that $\overline{f^{-1} (U)}\subset U$. Then $f$ has the rate.}\\

The proof relies on Brouwer's Theory for orientation preserving homeomorphisms of the plane.  We will use the following result:

\begin{thm}\label{cal}\cite{cl}  Let $f:\R^2\to \R^2$ be an orientation preserving homeomorphism and let $K\subset \R^2$ be an $f$-invariant non-separating continuum.  Then, there exists 
$x\in K$ such that $f(x) = x$.
 
\end{thm}

An easy proof or Cartwright-Littlewood's theorem can be found in  the  single-page paper of M.Brown \cite{brown}. Existence of a fixed point under the hypothesis that a compact 
subset is preserved was already known (Brouwer's plane translation theorem \cite{brou}). To prove that the fixed point must belong to the set $K$ one needs connectedness of such a set 
(just think of a rational rotation, where every periodic orbit is a compact invariant set disjoint from the set of fixed points).

\begin{proof} The hypothesis implies that $\overline{f^{-1} (U)}$ has $d$ connected components denoted $W_1,\ldots,W_d$, each $W_i$ closed and contained in $U$. Besides, 
$f(W_i) =\overline U$ for all $i$.

We will construct by induction on $n$ a sequence of $d^n$ sets $W^n_a$, indexed with $a$, sequences of $n$ elements between $1$ and $d$, where $W^1_i = W_i$.

Given $a\in \{1,\ldots,d\}^\N$ let $a|_n$ be the restriction of $a$ to the set $\{1,\ldots,n\}$

The induction hypothesis:\\ 
For each $j< n$ and $a\in \{1,\ldots,d\}^j$ there exists a set $W^j_a$ satisfying the following properties:
\begin{enumerate}
\item
$W^j_a$ is a compact connected subset of $W^{j-1}_{a'}$, where $a'$ is the restriction of $a$ to $\{1,\ldots,j-1\}$, i.e, $a'=(a_1,\ldots,a_{j-1})$.
\item
$W^{j-1}_{a'}=\cup_{i=1}^nW^{j}_{a'i}$ where $a'i$ is equal to $(a_1,\ldots,a_{j-1},i)$.
\item
$f(W^j_{ia'})=W^{j-1}_{a'}$ where $ia'$ is equal to $(i, a_1, \ldots, a_{j-1})$.
\end{enumerate}

Given $a=(a_1,\ldots,a_n)$, let $a''=(a_2,\ldots,a_{n})$, and define $W^n_a=f^{-1}(W^{n-1}_{a''})\cap W^1_{a_1}$.
As $W^{n-1}_{a'}$ is not empty and contained in $U$, then it has one $f$-preimage in each $W^1_i$. This implies the properties (1) to (3) above for $j=n$.

Then define $K_a$ for $a\in\{1,\ldots,d\}^\N$ as

$$
K_a=\cap_{n\geq 1} W^n_{a|_n}.
$$

Note that $K_a$ is a nonempty connected non-separating and compact subset as it is the decreasing intersection of compact connected sets. Moreover, if $a$ is a periodic element of $\{1,\ldots,d\}^\N$ with period $k$, then $K_a$ is $f^k$ invariant. Moreover, there 
exists a neighborhood $V$ of $K_a$ homeomorphic to a disc such that $f^j(V)\cap S_f$ is empty for every $j\leq k$, and it follows that $f^k$ restricted to $V$ is a homeomorphism onto 
its image. One can then extend $f^k|_V$ to a plane homeomorphism and apply Cartwright-Littlewood's Theorem \ref{cal} to obtain that $f^k$ has a fixed point in $K_a$. As the sets $K_a$ are
disjoint, and 
there are $d^k$ different sequences in $\{1,\ldots,d\}^\N$
having period $k$, the result follows.
 
\end{proof}

\end{document}